\documentclass[12pt, reqno]{amsart}
\usepackage{amsmath, amsthm, amscd, amsfonts, latexsym, amssymb, graphicx, color}
\usepackage[bookmarksnumbered, colorlinks, plainpages]{hyperref}
\usepackage[all, knot, arc, poly]{xy}

\makeatletter \oddsidemargin.9375in \evensidemargin \oddsidemargin
\newtheorem{thm}{Theorem}[section]
\newtheorem{lem}[thm]{Lemma}

\newtheorem{cor}[thm]{Corollary}
\theoremstyle{definition}
\newtheorem{defn}[thm]{Definition}
\newtheorem{exa}[thm]{Example}

\theoremstyle{remark}
\newtheorem{rem}[thm]{Remark}
\newtheorem{noti}[thm]{Notice}
\theoremstyle{fact}
\newtheorem{fac}[thm]{Fact}
\numberwithin{equation}{section}
\begin{document}
\setcounter{page}{1}

\title[First Non-abelian Cohomology of Topological Groups II]{First Non-abelian Cohomology of Topological Groups II}
\author[Sahleh and Koshkoshi]{H. Sahleh $^{*}$ and H.E. Koshkoshi}
\thanks{{\scriptsize
\hskip -0.4 true cm MSC(2010): Primary: 22A05, 20J06; Secondary: 18G50
$*$Corresponding author }}
\begin{abstract}
In this paper we introduce a new definition of the first non-abelian cohomology of topological groups.  We relate the cohomology of a normal subgroup $N$ of a topological group $G$ and the quotient $G/N$ to the cohomology of $G$. We get the inflation-restriction exact sequence. Also, we obtain a seven-term exact cohomology sequence up to dimension 2. We give an interpretation of the first non-abelian cohomology of a topological group by the notion of a principle homogeneous space.
\end{abstract}

\maketitle
\section{Introduction}\label{section 1}
 The first non-abelian cohomology of group $G$ with coefficients in a crossed $G$-module $(A,\mu)$ was (algebraically) introduced by Guin \cite{Gu1, Gu2}. The Guin's approach extended by H. Inassaridze to any dimension of non-abelian cohomology of $G$ with coefficients in a (partially) crossed topological $G-R$-bimodule $(A,\mu)$ (see \cite{Ina-h, Ina-n, Ina-c}). We generalize the Inassaridze's approach to define the first cohomology of non-abelain cohomology of topological groups. We continue to study non-abelian cohomology of topological groups (see \cite{Sah2, Sah1}).
\par  In this paper, topological groups are not necessarily abelian, unless otherwise specified. Let $G$  and $A$ be topological groups. It is said that $A$ is a topological $G$-module, whenever $G$   acts continuously on the  left of $A$. For all $g\in G$ and $a\in A$ we denote the action of $g$ on $a$ by $^{g}a$. If $A$ is a topological $G$-module, $H^{1}(G,A)$ denotes the first cohomology of $G$ with coefficients in $A$ \cite{Sah1}. The group isomorphism is denoted by  $\cong$. The center  and the commutator subgroup of a topological group $G$ is denoted by  $Z(G)$ and $[G,G]$, respectively. If the topological groups $G$ and $R$ act continuously on a topological group $A$, then the notation $^{gr}a$ means $^{g}(^{r}a)$, $g\in G$, $r\in R$, $a\in A$. We assume that every topological group acts on itself by conjugation. The coboundary map $\delta_{G}^{n}$ is as in  \cite[Definition (3.1)]{Hu}. If $G$ and $H$ are topological groups, then $\mathbf{1}:G\to H$ denotes the trivial homomorphism. If $f:G\to H$ is a map, then in what follows  $f^{-1}:G\to H$ denotes a map given by $f^{-1}(g)=f(g)^{-1}$. Note that if $P$ is a principle homogeneous space over a partially crossed topological $G-R$-bimodule $(A,\mu)$, then $e_{p}^{-1}:P\to A$ is the inverse map of the homeomorphism $e_{p}:A\to P$, $p\in P$ (see section \ref{section 5}).
\par In section \ref{section 2}, using the notion of partially crossed topological $G-R$-bimodule $(A,\mu)$ we define $Der_{c}(G,(A,\mu))$ as in \cite{Sah2} and the first non-abelian cohomology, $H^{1}(G,(A,\mu))$, of $G$  with coefficients in $(A,\mu)$ as a quotient of  $Der_{c}(G,(A,\mu))$ (see Definition \ref{Definition 1.10.}). Every topological $G$-module $A$ realises  a (partially) crossed topological $G-A/Z(A)$-bimodule $(A,\pi_{A})$. So, we may define the first cohomology, $\bar H^{1}(G,A)$, of $G$ with coefficients in $A$ as follows:
 $$\bar H^{1}(G,A)=H^{1}(G,(A,\pi_{A}))$$
 In general this  definition of the first non-abelian cohomology differs from the first pointed set cohomology, $H^{1}(G,A)$, which is defined in \cite{Sah1}. We show that there is a natural injection $\zeta:H^{1}(G,(A,\mu))\to H^{1}(G,A)$. Thus, there is an injection $\bar H^{1}(G,A)\to H^{1}(G,A)$. As a result, we may identify $\bar H^{1}(G,A)$ with $H^{1}(G,A)$, whenever $H^{1}(G,A/Z(A))$ is null. In particular, if $A$ is abelian, then $\bar H^{1}(G,A)\cong H^{1}(G,A)$.
\par In section \ref{section 3}, we introduce a new concept called  \emph{cocompatible triple} to study the change of groups. Using the cocompatible triple we define the inflation and the restriction maps and will show that for any normal subgroup $N$ of $G$, there is an exact sequence
$$\xymatrix{1\ar[r]&H^{1}(G/N,(A^{N},\mu^{N}))\ar[r]^{\ \ \ \ \  Inf^{1}}&H^{1}(G,(A,\mu))\ar[r]^{Res^{1} \ \ \ \ }& H^{1}(N,(A,\mu))^{G/N}}$$
In section \ref{section 4}, we show that for every proper extension, $1\to (A,\mathbf{1})\stackrel{\iota}\to (B,\mu)\stackrel{\pi}\to (C,\lambda)\to 1$, with continuous sections there exists a seven-term exact sequence,
\begin{center}
$0\to H^{0}(G,A)\stackrel{\iota^{0}}\to H^{0}(G,B)\stackrel{\pi^{0}}\to H^{0}(G,C)\stackrel{\delta^{0}}\to H^{1}(G,A)\stackrel{\iota^{1}}\to $
\end{center}
 $\ \ \ \ \ \ \ \ H^{1}(G,(B,\mu))  \stackrel{\pi^{1}}\to H^{1}(G,(C,\lambda))\stackrel{\delta^{1}}\to H^{2}(G,A).$
\par In section \ref{section 5}, we define  a principle homogeneous space (or topological torsor) over a partially crossed topological $G-R$-bimodule $(A,\mu)$ and we show that the first non-abelian cohomology of $G$ with coefficients in $(A,\mu)$  classifies the set of all  principle homogeneous spaces over $(A,\mu)$. We denote by $\mathcal{P}(A,\mu)$ the set of all classes of principle homogeneous spaces over $(A,\mu)$. Thus, naturally there exists a group product on $\mathcal{P}(A,\mu)$, whenever $(A,\mu)$ satisfying conditions of Theorem \ref{Theorem 1.23.}. As a result, if $(A,\mu)$ is a topological crossed $G$-module, then  $\mathcal{P}(A,\mu)$ is a group (not necessarily abelian).
\section{ Partially Crossed Topological Bimodules and the First Non-abelian Cohomology.}\label{section 2}
In this section, we introduce a partially crossed topological $G-R$-bimodule $(A,\mu)$ to define the first non-abelian cohomology, $H^{1}(G,(A,\mu))$, of $G$ with coefficients in $(A,\mu)$.
\begin{defn}\label{Definition 1.1.} A  precrossed topological $R$-module $(A,\mu)$ consists of a topological $R$-module $A$ and a continuous homomorphism $\mu:A\rightarrow R$ such that
\begin{center}
$\mu(^{r}a)=\ ^r\mu(a)$, $r\in R, a\in A.$
\end{center}
 If in addition we have
\begin{center}
$^{\mu(a)}b=\ ^{a}b$,  $a, b\in A$,
\end{center}
 then $(A,\mu)$ is called a crossed topological $R$-module.
 \end{defn}
\begin{defn}\label{Definition 1.2.} A precrossed topological $R$-module $(A,\mu)$ is said to be a partially crossed topological $R$-module, whenever it satisfies the following equality
$$^{\mu(a)}b=\ ^{a}b,$$ for all $b\in A$ and for all $a\in A$ such that $\mu(a)\in [R,R]$.
\end{defn}
 It is clear that every crossed topological $R$-module is  a partially crossed topological $R$-module.
\begin{defn}\label{Definition 1.4.} Let $G$, $R$ and $A$ be topological groups. Precrossed topological $R$-module $(A,\mu)$ is said to be  a precrossed topological $G-R$-bimodule, whenever
\begin{itemize}
\item[(1)] $G$ continuously acts on $R$ and $A$;
\item[(2)] $\mu:A\rightarrow R$ is a $G$-homomorphism;
\item[(3)] $^{(^gr)}a=\ ^{grg^{-1}}a$ for all $g\in G$, $r\in R$ and $a\in A$.
\end{itemize}
\end{defn}
\begin{defn}\label{Definition 1.5.} A precrossed topological $G-R$-bimodule $(A,\mu)$ is said to be a (partially) crossed topological $G-R$-bimodule, if  $(A,\mu)$ is a (partially) crossed topological $R$-module.
\end{defn}
\begin{exa}\label{Example 1.6.} (1) Let $A$ be an arbitrary topological G-module. Obviously, $Z(A)$ is a topological $G$-module. Now, we define an action of $R=A/Z(A)$ on $A$ and an
action of $G$ on $R$ by:
\begin{equation}\label{2.1}
 ^{aZ(A)}b=\ ^{a}b, \forall a,b\in A, \ \ \ \ ^{g}(aZ(A))=\ ^{g}aZ(A), \forall g\in G, a\in A.
\end{equation}
Let $\pi_{A}:A\rightarrow R$  be the canonical homomorphism. It is easy to see that under $(\ref{2.1})$ the pair
$(A,\pi_{A})$ is a  crossed topological $G-R$-bimodule.
\begin{itemize}
  \item[(2)] By part (1), for any topological group $G$ the pair $(G,\pi_{G})$ is a crossed topological $G-G/Z(G)$-bimodule.
\end{itemize}
\end{exa}
\begin{noti}\label{ Notice 1.7.} It is obvious that any precrossed (crossed or partially crossed) topological $R$-module is naturally viewed as a precrossed (crossed or partially crossed) topological $R-R$-bimodule.
\end{noti}
\begin{defn}\label{Definition 1.8.} A morphism $f:(A,\mu)\rightarrow (B,\nu)$ of precrossed (crossed) topological $G-R$-bimodule is a continuous homomorphism of topological groups $f:A\rightarrow B$ such that
\begin{itemize}
\item[(1)] $f(^{r}a)=\ ^{r}f(a)$, $r\in R$, $a\in A$;
\item[(2)] $f(^{g}a)=\ ^{g}f(a)$, $g\in G$, $a\in A$;
\item[(3)] $\mu=\nu\circ f$.
\end{itemize}
\end{defn}
\begin{defn}\label{Definition 1.9.} Let $(A,\mu)$ be a partially crossed topological $G-R$-bimodule. Denote by $Der(G,(A,\mu))$ the set of all pairs $(\alpha,r)$ where $\alpha$ is a crossed homomorphism from $G$ into $A$, i.e., $\alpha(gh)=\alpha(g)^{g}\alpha(h), \forall g, h\in G$; and $r$ is an element of $R$ such that $$\mu\circ \alpha(g)=\ r^{g}r^{-1}, \forall g\in G.$$
Let $Der_{c}(G,(A,\mu))$ be a subset of $Der(G,(A,\mu))$ which is defined as follows:
\begin{center}
 $Der_{c}(G,(A,\mu))=\{(\alpha,r)\in Der(G,(A,\mu))| \alpha$ is continuous$\}.$
\end{center}
 \end{defn}
 There is a group product $\star$ in $Der(G,(A,\mu))$  by $(\alpha,r)\star(\beta,s)=(\alpha*\beta,rs)$, where $\alpha*\beta(g)=\ ^{r}\beta(g)\alpha(g), \forall g\in G$ (see  \cite[Proposition 1.5] {Ina-h}). It is easy to see that $Der_{c}(G,(A,\mu))$ is a subgroup of $Der(G,(A,\mu))$.
\\
\par Let $R$ be a topological $G$-module, then we define $$H^{0}(G,R)=\{r|^{g}r=r,\forall g\in G\}.$$
\par Let $(A,\mu)$ be a partially crossed topological $G-R$-bimodule. H. Inassaridze \cite{Ina-h} introduced the equivalence relation $\sim$  on the group $Der(G,(A,\mu))$ as
$(\alpha,r)\sim(\beta,s)$ whenever
\begin{equation}\label{2.2}
\exists \ a\in A \wedge ( \forall g\in G\Rightarrow\beta(g)=a^{-1}\alpha(g)^{g}a
\end{equation}  and
\begin{equation}\label{2.3}
 s=\mu(a)^{-1}r\  mod\  H^{0}(G,R)
 \end{equation}
\par Let $\sim'$ be the restriction of $\sim$ to $Der_{c}(G,(A,\mu))$. Therefore, $\sim'$ is an equivalence relation. In other word, $(\alpha,r)\sim'(\beta,s)$ if and only if $(\alpha,r)\sim(\beta,s)$, whenever $(\alpha,r), (\beta,s)\in Der_{c}(G,(A,\mu))$.
\\ \par Let $A$ be a  topological $G$-module. We denote by $Inn(G,A)$ the set of all inner crossed homomorphisms from $G$ into $A$, in other words: $$Inn(G,A)=\{Inn(a)|a\in A, Inn(a)(g)=a^{g}a^{-1}\}$$
Similarly, if $(A,\mu)$ is a partially crossed topological $G-R$-bimodule, then put
$$Inn(G,(A,\mu))=\{(Inn(a),\mu(a)z)| a\in A, z\in H^{0}(G,R)\}.$$
\begin{defn}\label{Definition 1.10.} Let $(A,\mu)$ be a partially crossed topological $G-R$-bimodule. The quotient set  $Der_{c}(G,(A,\mu))/\sim'$  will be called the first cohomology  of $G$ with the coefficients in $(A,\mu)$ and is denoted by $H^{1}(G,(A,\mu))$.
\end{defn}
\begin{fac}\label{ Fact 1.10.} Let $(A,\mu)$ is a partially crossed topological $G-R$-bimodule. Then
 \begin{itemize}
   \item[(1)] $^{zg}a=\ ^{gz}a$,  $\forall a\in A$, $g\in G$, $z\in H^{0}(G,R)$.
   \item[(2)] $\alpha(g)^{gr}a=\ ^{rg}a\alpha(g)$, $\forall g\in G, a\in A, (\alpha,r)\in Der_{c}(G,(A,\mu))$.
 \end{itemize}
 \end{fac}
\begin{proof} See \cite[p. 499]{Ina-h} and \cite[p. 315]{Ina-n}.
\end{proof}
\begin{defn}\label{Definition 1.11.} Let $A$ be a topological $G$-module. Denote by $\bar H^{1}(G,A)$ the first cohomology of $G$ with coefficients in $A$ and define it as follows:
 $$\bar H^{1}(G,A)=H^{1}(G,(A,\pi_{A})).$$
 \end{defn}
Let $(A,\mu)$ be a partially crossed topological $G-R$-bimodule, then there is the following natural map
\begin{center}
$\zeta:H^{1}(G,(A,\mu))\to H^{1}(G,A)$
$$[(\alpha,r)]\mapsto [\alpha].$$
\end{center}
\begin{thm}\label{Theorem 1.12.} Let $(A,\mu)$ be a partially crossed topological $G-R$-bimodule. Then
 \begin{itemize}
   \item[(i)] $H^{1}(G,(A,\mu))$ can be naturally embedded in $H^{1}(G,A)$.
   \item[(ii)] There is a bijection between $H^{1}(G,(A,\mu))$ and $H^{1}(G,A)$ if and only if the induced map $\mu^{1}:H^{1}(G,A)\to H^{1}(G,R)$ is trivial (that is, $\mu^{1}=\mathbf{1}$).
 \end{itemize}
 \end{thm}
\begin{proof} Let $\zeta$ be the natural map mentioned above. (i). Suppose that $\zeta[(\alpha,r)])=\zeta([(\beta,s)])$. Hence there is $a\in A$ such that $\beta(g)=a^{-1}\alpha(g)^{g}a$, for all $g\in G$. Hence, $\mu\beta(g)=\mu(a)^{-1}\mu\alpha(g)^{g}\mu(a)$. Thus, $s^{g}s^{-1}=\mu(a)^{-1}r^{g}{r}^{-1}\ ^{g}\mu(a)$, for all $g\in G$. Therefore for all $g\in G$,  $^{g}(r^{-1}\mu(a)s)=r^{-1}\mu(a)s$, i.e., $r^{-1}\mu(a)s\in H^{0}(G,R)$. This shows that $\zeta$ is one to one.
\par (ii).  Let $\alpha\in Der_{c}(G,A)$, then $\mu\alpha\in Der_{c}(G,R)$. If $\mu^{1}=\mathbf{1}$, then there is an $r\in R$ such that for all $g\in G$, $\mu\alpha(g)=r^{g}r^{-1}$.  Thus $\zeta([(\alpha,r)])=[\alpha]$. So, $\zeta$ is onto. Conversely, if $\zeta$ is onto. Let $\alpha\in Der_{c}(G,A)$, then  there is $(\beta,r)\in Der_{c}(G,(A,\mu))$ such that $\zeta([(\beta,r)])=[\alpha]$. Thus, there is $a\in A$ such that $\alpha(g)=a^{-1}\beta(g)^{g}a$, whence $\mu^{1}([\alpha])=\mu^{1}([\beta])=[\mu\beta]=[Inn(r)]=1$, and this finishes the proof.
\end{proof}
\begin{cor}\label{Corollary 1.13.} Let $(A,\mu)$ be a partially crossed topological $G-R$-bimodule and $H^{1}(G,R)=0$. Then, there is a bijection between $H^{1}(G,A)$ and $H^{1}(G,(A,\mu))$.
\end{cor}
\begin{proof} The proof is immediate.
\end{proof}
Theorem \ref{Theorem 1.12.} implies immediately the following two corollaries:
\begin{cor}\label{Corollary 1.14.} Let $A$ be a topological $G$-module. There is a bijection between  $\bar H^{1}(G,A)$  and $H^{1}(G,A)$ if and only if $\pi^{1}_{A}:H^{1}(G,A)\to H^{1}(G,A/Z(A))$ is null.
\end{cor}
\begin{cor}\label{Corollary 1.15.} If $A$ is a topological $G$-module and $H^{1}(G,A/Z(A))=0$, then there is a bijection between  $\bar H^{1}(G,A)$  and $H^{1}(G,A)$. In particular if $A$ is abelian, then $\bar H^{1}(G,A)\cong H^{1}(G,A)$.
\end{cor}
 According to part (i) of the proof of Theorem \ref{Theorem 1.12.}, we have the next remark.
\begin{rem}\label{Remark 1.16.}  Let $(A,\mu)$ be a partially crossed (topological) $G-R$-bimodule, and let $(\alpha,r),(\beta,s)\in Der(G,(A,\mu))$. Then
\begin{center}
$(\alpha,r)\sim (\beta,s) \Leftrightarrow \exists a\in A \wedge (\forall g\in G \Rightarrow \beta(g)=a^{-1}\alpha(g)^{g}a)$.
\end{center}
In other words, the condition (\ref{2.2}) implies the condition (\ref{2.3}). Thus, we can delete the condition (\ref{2.3}) of the equivalence relation $\sim$.
\end{rem}
\begin{defn}\label{Definition 1.17.} Let $(A,\mu)$ be a partially crossed topological $G-R$-bimodule. The quotient set  $Der_{c}(G,(A,\mu))/\sim'$  will be called the first cohomology  of $G$ with the coefficients in $(A,\mu)$ and is denoted by $H^{1}(G,(A,\mu))$.
\end{defn}
\begin{thm}\label{ Theorem 1.18.} Let $(A,\mu)$ be a  crossed topological $G-R$-bimodule. Then, $Inn(G,(A,\mu))$ is a subgroup of $Der_{c}(G,(A,\mu))$.
\end{thm}
\begin{proof} Suppose that $(Inn(a),\mu(a)z)), (Inn(a'),\mu(a')z')\in Inn(G,(A,\mu))$. We get $(Inn(a),\mu(a)z))(Inn(a'),\mu(a')z')= (Inn(a^{z}a'),\mu(a^{z}a')zz')$, since for any $g\in G$
\begin{center}
$^{\mu(a)z}Inn(a')Inn(a)(g)=\ ^{\mu(a)z}Inn(a')(g)Inn(a)(g)$
$$=a^{z}Inn(a')(g)a^{-1}Inn(a)(g)
=a^{z}(a'^{g}a'^{-1})a^{-1}(a^{g}a^{-1})$$
$=a^{z}a'^{zg}{a'^{-1}}^{g}a^{-1}
=a^{z}a'^{gz}{a'^{-1}}^{g}a^{-1}
=a^{z}a'^{g}(^{z}a'^{-1}a^{-1})$
$$=a^{z}a'^{g}(a^{z}a')^{-1}=Inn(a^{z}a')(g),$$
\end{center}
and also, $\mu(a)z\mu(a')z'=\mu(a)z\mu(a')z^{-1}zz'=\mu(a)^{z}\mu(a')zz'=\mu(a^{z}a')zz'$.
Thus, $Inn(G,(A,\mu))$ is closed under multiplication.
\par In addition, We have $(Inn(a),\mu(a)z)^{-1}=(Inn(^{z^{-1}}a^{-1}),\mu(^{z^{-1}}a^{-1})z^{-1})$, since for any $g\in G$
\begin{center}
  $^{(\mu(a)z)^{-1}}(Inn(a)(g))^{-1}=\ ^{z^{-1}\mu(a^{-1})}(a^{g}a^{-1})^{-1}=\ ^{z^{-1}\mu(a^{-1})}(^{g}aa^{-1})=\ ^{z^{-1}}({a^{-1}}^{g}aa^{-1}a)=
  \ ^{z^{-1}}({a^{-1}}^{g}a)=\ ^{z^{-1}}{a^{-1}} ^{z^{-1}g}a=\ ^{z^{-1}}{a^{-1}}^{gz^{-1}}a= \ ^{z^{-1}}{a^{-1}}^{g}(^{z^{-1}}a)=Inn(^{z^{-1}}a^{-1})(g),$
\end{center}
and also, $(\mu(a)z)^{-1}=z^{-1}\mu(a^{-1})=\mu(^{z^{-1}}a^{-1})z^{-1}$.
Therefore, $Inn(G,(A,\mu))$ is closed under inversion. So, $Inn(G,(A,\mu))$ is a subgroup of $Der_{c}(G,(A,\mu))$.
\end{proof}
\begin{rem}\label{ Remark 1.19.} Note that $Inn(G,(A,\mu))$ is not necessarily a normal subgroup of $Der_{c}(G,(A,\mu))$. If $H^{1}(G,(A,\mu))$ is a group, then $Inn(G,(A,\mu))$ is a normal subgroup of $Der_{c}(G,(A,\mu))$ and $$H^{1}(G,(A,\mu))\cong Der_{c}(G,(A,\mu))/Inn(G,(A,\mu)).$$
\end{rem}
\begin{rem}\label{ Remark 1.20.} Let $(A,\mu)$ be a crossed topological $G-R$-bimodule. Then, $Inn(G,(A,\mu))$ is a normal subgroup of $Der_{c}(G,(A,\mu))$ if and only if $^{(^{r}z)}\alpha^{-1}\alpha\in Inn(G,A)$ for all $(\alpha,r)\in Der_{c}(G,(A,\mu)), z\in H^{0}(G,R)$.
\end{rem}
\begin{proof} Suppose that $(\alpha,r)\in Der_{c}(G,(A,\mu))$ and $(Inn(a),\mu(a)z)\in Inn(G,(A,\mu))$. We have
\begin{center}
  $(\alpha,r)(Inn(a),\mu(a)z)(\alpha,r)^{-1}=(^{r\mu(a)zr^{-1}}{\alpha^{-1}}^{r}Inn(a)\alpha,r\mu(a)zr^{-1})$
\end{center}
Thus, $$(^{r\mu(a)zr^{-1}}{\alpha^{-1}}^{r}Inn(a)\alpha)(g)=\ ^{\mu(^{r}a)(^{r}z)}{\alpha^{-1}(g)}^{r}Inn(a)(g)\alpha(g)$$
\begin{center}
$=\ ^{r}a^{(^{r}z)}{\alpha^{-1}(g)}^{rg}a^{-1}\alpha(g)=\ ^{r}a^{(^{r}z)}{\alpha^{-1}(g)}\alpha(g)^{gr}a^{-1}.$
\end{center}
The last equality is obtained by part (2) of  Fact \ref{ Fact 1.10.}. Hence, we have $^{r\mu(a)zr^{-1}}{\alpha^{-1}}^{r}Inn(a)\alpha\sim'\ ^{(^{r}z)}{\alpha^{-1}}\alpha $. This completes the proof.
 \end{proof}
As a consequence of Remark \ref{ Remark 1.20.} we get the following corollary.
\begin{cor}\label{ Corollary 1.21.} Let $(A,\mu)$ be a crossed topological $G-R$-bimodule and $A$ a trivial $R$-module. Then $Inn(G,(A,\mu))$ is a normal subgroup of $Der_{c}(G,(A,\mu))$.
\end{cor}
Let $(A,\mu)$ be a partially crossed topological $G-R$-bimodule. On can see there is a natural action of $H^{0}(G,R)$ on $H^{1}(G,(A,\mu))$ as follows:
\begin{center}
$^{z}[(\alpha,r)]=[(^{z}\alpha,^{z}r)]$, $z\in H^{0}(G,R),[(\alpha,r)]\in H^{1}(G,(A,\mu))$,
\end{center}
where  $(^{z}\alpha)(g)=\ ^{z}\alpha(g)$, for all $g\in G$.  Note that by part (1) of Fact \ref{ Fact 1.10.}, $^{z}\alpha$ is a crossed homomorphism.
\begin{lem}\label{ Lemma 1.22.} Let $(A,\mu)$ be a partially crossed topological $G-R$-bimodule. If $Der(G,(A,\mu))/\sim$ is a group, then $H^{1}(G,(A,\mu))$ is isomorphic to a subgroup of $Der(G,(A,\mu))/\sim$.
\end{lem}
\begin{proof} Clearly, the natural map $\jmath:H^{1}(G,(A,\mu))\to Der(G,(A,\mu))/\sim$, $[(\alpha,r)]\mapsto cls(\alpha,r)$ is injective. The equivalence relation $\sim'$ is  congruence, since by assumption, $\sim$ is  congruence. Thus, $\jmath$ is a homomorphism. This completes the proof.
\end{proof}
Indeed in Lemma \ref{ Lemma 1.22.}, $Der(G,(A,\mu))/\sim$ is the first non-abelian cohomology of $G$ with coefficients in $(A,\mu)$ \cite{Ina-h, Ina-c}.
\begin{thm}\label{Theorem 1.23.} Let $(A,\mu)$ be a partially crossed topological $G-R$-bimodule satisfying
the following conditions	
\begin{itemize}
  \item[(i)] $H^{0}(G,R)$ is a normal subgroup of $R$;
  \item[(ii)] for every $c\in H^{0}(G,R)$ and $(\alpha,r)\in Der_{c}(G,(A,\mu))$, there exists
$a \in Ker\mu$ and $^{c}\alpha(g)=a^{-1}\alpha(g)^{g}a$, $\forall g\in G$.
\end{itemize}
Then, $Der_{c}(G,(A,\mu))$ induces  a  group structure on $H^{1}(G,(A,\mu))$.
\end{thm}
\begin{proof} By  \cite[Theorem 2.1]{Ina-h}, the quotient set $Der(G,(A,\mu))/\sim$ is a group. Thus, by Lemma \ref{ Lemma 1.22.}, $H^{1}(G,(A,\mu))$ is a group.  \end{proof}
Let $A$ and $B$ be topological $G$-modules and let $\mu:A\to B$ be a continuous $G$-homomorphism. We say that $\mu$ is a $G$-retraction whenever there is a continuous $G$-homomorphism $\rho:B\to A$ such that $\mu\rho=Id_{B}$. For example, $(G,Id_{G})$ is a crossed topological $G$-module and clearly, $Id_{G}:G\to G$ is a $G$-retraction.
\begin{thm}\label{ Theorem 1.24.} Let $(A,\mu)$ be a partially crossed topological $G-R$-bimodule and suppose that $\mu:A\to R$ is a $G$-retraction. Then, the following is an exact sequence.
$$1\to H^{1}(G,(A,\mu))\stackrel {\zeta} \to H^{1}(G,A)\stackrel{\mu^1}\to H^{1}(G,R)\to 1 $$
\end{thm}
\begin{proof} By  part (i) of Theorem \ref{Theorem 1.12.}, $\zeta$ is one to one. If $(\alpha,r)\in Der_{c}(G,(A,\mu))$, then $\mu^1\zeta([(\alpha,r)])=\mu^1([\alpha])=[\mu\circ\alpha]=[Inn(r)]=1$. Thus $Im\zeta\subset Ker\mu^1$. Vice versa, if $[\alpha]\in Ker\mu^1$, then $\mu\alpha$ is cohomologous to $\mathbf{1}$. Hence, there is $r\in R$ such that $\mu\alpha(g)=r^{g}r^{-1}$, for all $g\in G$. So, $(\alpha,r)\in Der_{c}(G,(A,\mu))$ and $\zeta([(\alpha,r)])=[\alpha]$. Therefore, $Ker\mu^1\subset Im\zeta$. Finally, we show that $\mu^1$ is onto. Suppose that $\alpha\in Der_{c}(G,R)$. Set $\beta=\rho\alpha$. Obviously, $\beta\in Der_{c}(G,A)$ and $\mu^1([\beta])=[\alpha]$.
\end{proof}
\begin{thm}\label{ Theorem 1.25.} Let $(A,\mu)$ be a partially crossed topological $G-R$-bimodule satisfying the following conditions:
\begin{enumerate}
  \item[(1)] $A$ and $R$ are abelian;
  \item[(2)] for any $r\in R$ and $(\alpha,s)\in Der_{c}(G,(A,\mu))$ there exists $a\in Ker\mu$ and $^{r}\alpha(g)=a^{-1}\alpha(g)^{g}a$;
  \item[(3)] $\mu:A\to R$  is a $G$-retraction.
\end{enumerate}
 Then, $H^{1}(G,A)\cong H^{1}(G,(A,\mu))\oplus H^{1}(G,R)$.
 \end{thm}
\begin{proof} The condition (1) implies that $H^{1}(G,A)$ and $H^{1}(G,R)$ are abelian groups. By Theorem \ref{Theorem 1.23.}, $H^{1}(G,(A,\mu))$ is a group. The conditions (1) and (2) imply that the map $\zeta$ is a homomorphism, since
\begin{center}
$\zeta([(\alpha,r)[(\beta,s)])=\zeta([\alpha^{r}\beta],rs)])=[\alpha^{r}\beta]=[\alpha][^{r}\beta]=[\alpha][\beta]=\zeta([(\alpha,r)])\zeta([\beta,s)]).$
\end{center}
Hence, by Theorem \ref{ Theorem 1.24.}, $1\to H^{1}(G,(A,\mu))\stackrel {\zeta} \to H^{1}(G,A)\stackrel{\mu^1}\to H^{1}(G,R)\to 1 $ is an exact sequence (of groups and homomorphisms). By (3), there is a continuous $G$-homomorphism $\rho$ such that $\mu\rho=Id_{R}$. Thus, $\mu^1\rho^1=(\mu\rho)^1=(Id_{R})^1=Id_{H^{1}(G,R)}$. This completes the proof.
\end{proof}
 As an immediate result of  Theorem \ref{ Theorem 1.25.}, we have:
\begin{cor}\label{Corollary 1.26.} Let $(A,\mu)$ be a partially crossed topological $G$-module satisfying the following conditions:
\begin{itemize}
  \item[(1)] $G$ and $A$ are abelian;
  \item[(2)] $\mu:A\to G$ is a $G$-retraction.
\end{itemize}
Then, $H^{1}(G,A)\cong H^{1}(G,(A,\mu))\oplus Hom_{c}(G,G)$.
\end{cor}
Note that if $A$ is an abelian topological $G$-module, then $(\alpha,r)\sim'(\alpha,1)$, for every $(\alpha,r)\in Der_{c}(G,(A,\mathbf{1}))$. Therefore, we have the next theorem.
\begin{thm}\label{Theorem 1.27.} Let $G$ be a topological group. Then,
 $$\tau_{A}: H^{1}(G,A)\rightarrow H^{1}(G,(A,\mathbf{1})),\ \tau_{A}([\alpha])=[(\alpha,1)]$$
is a natural isomorphism in the category of abelian topological $G$-modules.
\end{thm}
\begin{proof} The proof is an standard argument.
\end{proof}
\section{ Change of Groups for the First Cohomology.}\label{section 3}
We introduce a notion called cocompatible triple and we get inflation-restriction exact sequence for the first non-abelian cohomology groups.
\begin{defn}\label{Definition 2.1.} Let $(A,\mu)$ be a partially crossed topological $G-R$-bimodule and $(A',\mu')$ a partially crossed topological $G'-R'$-bimodule. Suppose that $\phi:G'\to G$, $\varphi:R\to R'$ and $\psi:A\to A'$ are continuous homomorphisms. The triple $(\phi,\varphi,\psi)$ is called a cocompatible triple whenever the following conditions hold:
\begin{enumerate}
  \item $^{g'}\varphi(r)=\varphi(^{\phi(g')}r)$, $\forall g'\in G', r\in R$;
  \item $^{g'}\psi(a)=\psi(^{\phi(g')}a)$, $\forall g'\in G', a\in A$.
\end{enumerate}
\end{defn}
\begin{exa}\label{Example 2.1.} If $(A,\mu)$ is a partially crossed topological $G-R$-bimodule and $N$ a subgroup of $G$. Then, $(A,\mu)$ is a partially crossed $N-R$-bimodule. The triple
$(\imath, Id_{R}, Id_{A})$ is a cocompatible triple, where $\imath:N\to G$ is the inclusion map and $Id_{R}$ and $Id_{A}$ are the identity maps.
\end{exa}
\begin{exa}\label{Example 2.2.} If $N$ is a normal subgroup of $G$ and $\mu^{N}:A^{N}\to R^{N}$ is the restriction of $\mu:A\to R$. Clearly $(A^{N},\mu^{N})$ is a partially crossed topological $G/N-R^{N}$-bimodule. The triple $(\pi, \imath, \jmath)$ is a cocompatible triple, where $\pi:G\to G/N$ is the natural map, $\imath:R^{N}\to R$ and $\jmath: A^{N}\to A$ are the inclusion maps.
\end{exa}
Note that a cocompatible triple $(\phi,\varphi,\psi)$ induces a natural map as follows:
\begin{center}
$Der_{c}(G,(A,\mu))\to Der_{c}(G',(A',\mu'))$, $(\alpha,r)\mapsto (\psi\circ\alpha\circ\phi,\varphi(r))$
\end{center}
which induces naturally the map:
\begin{center}
$(\phi,\varphi,\psi)^{1}: H^{1}(G,(A,\mu))\to H^{1}(G',(A',\mu'))$, $[(\alpha,r)]\mapsto [(\psi\circ\alpha\circ\phi,\varphi(r))]$.
\end{center}
\begin{defn}\label{Definition 2.2.} Let $(A,\mu)$ be a partially crossed topological $G-R$-bimodule and $N$ a subgroup of $G$. The induced map $(\imath,Id_{R},Id_{A})^{1}$ is
called the restriction map and it is denoted by $Res^{1}: H^{1}(G,(A,\mu))\to H^{1}(N,(A,\mu))$.
\end{defn}
\begin{defn}\label{Definition 2.3.} Let $(A,\mu)$ be a partially crossed topological $G-R$-bimodule and $N$ a normal subgroup of $G$. The induced map $(\pi,\imath,\jmath)^{1}$ is
called the inflation map and it is denoted by $Inf^{1}: H^{1}(G/N,(A^{N},\mu^{N}))\to H^{1}(G,(A,\mu))$.
\end{defn}
\begin{lem}\label{Lemma 2.4.} Let $(A,\mu)$ be a partially crossed topological $G-R$-bimodule, and $N$ a normal subgroup of $G$. Then,
\begin{itemize}
\item[(i)]  $H^{1}(N,(A,\mu))$ is a $G/N$-set. Moreover, if $H^{1}(N,(A,\mu))$ is a group, then $H^{1}(N,(A,\mu))$ is a $G/N$-module.

\item[(ii)] $ImRes^{1}\subset H^{1}(N,(A,\mu))^{G/N}$.

\end{itemize}
\end{lem}
\begin{proof} (i) Since $N$ is a normal subgroup of $G$, then, there is an action of $G$ on $Der_{c}(N,(A,\mu))$ as follows:\\
 For every $g\in G$ we define $^{g}(\alpha,r)=(\tilde{\alpha},^{g}r)$ with $\tilde{\alpha}(n)=\-^{g}\alpha(^{g^{-1}}n), n\in N$.
\par In fact, $\tilde{\alpha}$ is continuous and we have:
\begin{center}
    $\tilde{\alpha}(mn)=\-^{g}\alpha(^{g^{-1}}(mn))=\-^{g}\alpha(^{g^{-1}}m^{g^{-1}}n)=
    \-^{g}\alpha(^{g^{-1}}m)\-^{mg}\alpha(^{g^{-1}}n)=\tilde{\alpha}(m)^{m}\tilde{\alpha}(n)$,
\end{center}
whence, $\tilde{\alpha}\in Der_{c}(N,A)$. Also it is easy to see that $\mu(\tilde{\alpha}(n))=(^{g}r)^{n}(^{g}r^{-1})$, for every $n\in N$. Hence, $^{g}(\alpha,r)\in Der_{c}(N,(A,\mu))$. It is clear that $^{gh}(\alpha,r)=\ ^{g}(^{h}(\alpha,r))$. It is easy to verify that  $^{g}((\alpha,r)(\beta,s))=\-^{g}(\alpha,r)^{g}(\beta,s)$. Consequently $Der_{c}(N,(A,\mu))$ is a $G$-module. Now suppose that $(\alpha,r)\sim(\beta,s)$. Then, there is an $a\in A$ such that $\beta(n)=a^{-1}\alpha(n)^{n}a, \forall n\in N$. Thus, for every $g\in G$, $n\in N$,
\begin{center}
 $^{g}\beta(^{g^{-1}}n)=\-^{g}a^{-1}(^{g}\alpha(^{g^{-1}}n))^{g}(^{^{g^{-1}}n}a)$.
\end{center}
Therefore,
\begin{center}
    $\tilde{\beta}(n)=(^{g}a)^{-1}\tilde{\alpha}(n)^{n}(^{g}a)$.
\end{center}
Therefore, by Remark \ref{Remark 1.16.}, $^{g}(\alpha,r)\sim \ ^{g}(\beta,s)$.  Thus, the action of $G$ on $Der_{c}(N,(A,\mu))$ induces an action of $G$ on $H^{1}(N,(A,\mu))$. Moreover if $H^{1}(N,(A,\mu))$ is a group, then $H^{1}(N,(A,\mu))$ is a $G$-module. It is sufficient to show for every $m\in N$, $^{m}(\alpha,r)\sim (\alpha,r)$. In fact, for every $n\in N$
\begin{center}
    $^{m}\alpha(^{m^{-1}}n)=\-^{m}\alpha(m^{-1}nm)=\ ^{m}(\alpha(m^{-1})^{m^{-1}}\alpha(n)^{m^{-1}n}\alpha(m))
    =\ ^{m}\alpha(m^{-1})\alpha(n)^{n}\alpha(m)=\alpha(m)^{-1}\alpha(n)^{n}\alpha(m)$.
\end{center}
Thus, $H^{1}(N,(A,\mu))$ is a $G/N$-set. Moreover if $H^{1}(N,(A,\mu))$ is a group, then $H^{1}(N,(A,\mu))$ is a $G/N$-module.
\par (ii) By a similar argument as in (i), for every $(\alpha,r)\in Der_{c}(G,(A,\mu))$
\begin{center}
    $^{g}\alpha(^{g^{-1}}n)=\alpha(g)^{-1}\alpha(n)^{n}\alpha(g)$, $\forall g\in G, n\in N$,
\end{center}
whence, $^{gN}(\alpha\circ\imath,r)\sim (\alpha\circ\imath,r), \forall gN\in G/N$.
 \end{proof}
\begin{thm}\label{Theorem 2.5.} Let $(A,\mu)$ be a partially crossed topological $G-R$-bimodule and $N$  a normal subgroup of $G$. Then,  there is an exact sequence
\begin{center}
    $ $\xymatrix{1\ar[r]&H^{1}(G/N,(A^{N},\mu^{N}))\ar[r]^{\ \ \ \ \  Inf^{1}}&H^{1}(G,(A,\mu))\ar[r]^{Res^{1} \ \ \ \ }& H^{1}(N,(A,\mu))^{G/N}}$.$
\end{center}
\end{thm}
\begin{proof} The map $Inf^{1}$ is one to one: If $(\alpha,r), (\beta,s)\in Der_{c}(G/N,A^{N})$ and $Inf^{1}[(\alpha,r)]=Inf^{1}[(\beta,s)]$, then $(\alpha\pi,r)\sim (\beta\pi,s)$. Thus, there is an $a\in A$ such that $\beta\pi(g)=a^{-1}\alpha\pi(g)^{g}a, \forall g\in G$. Hence, $\beta(gN)=a^{-1}\alpha(gN)^{g}a,  \forall gN\in G/N$. If $g\in N$, then $\alpha(gN)=\beta(gN)=1$ and hence, $a\in A^{N}$. This implies that $^{g}a=\-^{(gN)}a, \forall g\in G$. Consequently, $(\alpha,r) \sim (\beta,s)$, i.e., $Inf^{1}$ is one to one.
\par Now we show that  $KerRes^{1}= ImInf^{1}$. Since $Res^{1}Inf^{1}[(\alpha,r)]=[(\alpha(\pi\imath),r)]=[(\mathbf{1},1)]$, then $ImInf^{1}\subset KerRes^{1}$.
\par Let $[(\alpha,r)] \in KerRes^{1}$. Then, there is an $a\in A$ such that $\alpha(n)=a^{-1}$$^{n}a$, $\forall n\in N$. Consider $(\beta,\mu(a)r)\in Der_{c}(G,(A,\mu))$ with $\beta(g)=a\alpha(g)^{g}a^{-1}$, $\forall g\in G$. Since $\beta(n)=1, \forall n\in N$ then, $\beta$  induces  the continuous crossed homomorphism $\gamma:G/N\rightarrow A$ via $\gamma(gN)=\beta(g)$. Also $Im\gamma\subset A^{N}$, since for all $n\in N$,
\begin{center}
    $^{n}\gamma(gN)=\-^{n}\beta(g)=\beta(ng)=\beta(g)^{g}\beta(g^{-1}ng)=\beta(g)=\gamma(gN)$.
\end{center}
Clearly, $\mu(a)r\in H^{0}(N,R)$ and $(\gamma,\mu(a)r)\in Der_{c}(G/N,(A^{N},\mu^{N}))$. Hence, $Inf^{1}[(\gamma,\mu(a)r)]=[(\gamma\pi,\mu(a)r)]=[(\beta,\mu(a)r)]=[(\alpha,r)]$. Consequently, $KerRes^{1}\subset ImInf^{1}$.
\end{proof}
%
%
%
%
%
%
%
%
%
%
%
%
%
%
%
%
%
\section{ Coboundary Maps and Exact Sequence of Cohomologies.}\label{section 4}
 In this section we will obtain a seven-term exact sequence of non-abelian cohomologies up to dimension 2.
   \par Suppose that $1\to (A,\mathbf{1})\stackrel{\iota}\to (B,\mu)\stackrel{\pi}\to (C,\lambda)\to 1$  is an exact sequence of partially crossed topological $G-R$-bimodules such that $\iota$ is an homeomorphic embedding. Thus, we can identify $A$ with $\iota(A)$.
  \par
 Now we define a coboundary map $\delta^{0}:H^{0}(G,C)\rightarrow H^{1}(G,A)$.
 \\Let $c\in H^{0}(G,C)$,  $b\in B$ with $\pi(b)=c$. Then, we define $\delta^{0}(c)$ by $\delta^{0}(c)(g)=b^{-1}$$^{g}b, \forall g\in G$.
 It is obvious that $\delta^{0}(c)$ is a continuous crossed homomorphism. Let $b^{'} \in B$, $\pi(b^{'})=c$. Then, $b^{'}=ba$ for some $a\in A$. So,
\begin{center}
$(b^{'})^{-1}$$^{g}b^{'}=a^{-1}b^{-1}$$^{g}b^{g}a=a^{-1}\delta^{0}(c)(g)^{g}a$.
\end{center}
Thus, the crossed homomorphism obtained from $b{'}$ is cohomologous in $A$ to the one obtained from $b$,
i.e., $\delta^{0}$ is well-defined.
\\ \par Now,  suppose that $1\to (A,\mathbf{1})\stackrel{\iota}\to (B,\mu)\stackrel{\pi}\to (C,\lambda)\to 1$  is an exact sequence of partially crossed homomorphism such that $\iota$ is a homeomorphic embedding and in addition $\pi$ has a continuous section $s:C\rightarrow B$.
 \par We construct a coboundary map $\delta^{1}:H^{1}(G,(C,\lambda))\to H^{2}(G,A)$. Here, $H^{2}(G,A)$ is defined as in \cite{Hu}.
 \\ Let $\alpha \in H^{1}(G,(C,\lambda))$ and $s:C\rightarrow B$ be a continuous section for $\pi$. Define $ \delta^{1}$ by $[(\alpha,r)]\mapsto [\tilde{\alpha}]$, where $\tilde{\alpha}(g,h)=s\alpha(g)$ $^{g}(s\alpha(h))(s\alpha(gh))^{-1}$. Clearly, $ \tilde{\alpha}$ is a continuous map.
 \par We show that $\tilde\alpha$ is a factor set with values in $A$,
   and independent of the choice of the continuous section $s$. Also $\delta^{1}$ is well-defined.
 \\Since $\alpha$ is a crossed homomorphism, we get:
 \begin{center}
 $\pi(\tilde\alpha(g,h))=\pi($$s\alpha(g)$ $^{g}s\alpha(h)(s\alpha(gh))^{-1}=\alpha(g)^{g}\alpha(h)(\alpha(gh))^{-1}=1$.
 \end{center}
 Thus, $\tilde\alpha$ has values in $A$.
 \par Next, we show that $\tilde\alpha$ is a factor set, i.e.,
\begin{equation}\label{*}
^{g}\tilde\alpha(h,k)\tilde\alpha(g,hk)=\tilde\alpha(gh,k)\tilde\alpha(g,h), \ \forall g, h, k \in G.
\end{equation}
First we calculate the left hand side of  (\ref{*}). For simplicity, take $b_{g}=s\alpha(g)$, $\forall g\in G$. Since $A\subset Ker\mu$, then
\begin{center}
 $^{g}\tilde\alpha(h,k)\tilde\alpha(g,hk)=\-^{g}(b_{h}\-^{h}b_{k}b_{hk}^{-1})(b_{g}\-^{g}b_{hk}b_{ghk}^{-1})=
 b_{g}\-^{g}(b_{h}\-^{h}b_{k}b_{hk}^{-1})\-^{g}b_{hk}b_{ghk}^{-1}
$ \end{center}
\begin{center}
$=b_{g}\-^{g}(b_{h}\-^{h}b_{k})^{g}b_{hk}b_{ghk}^{-1}=b_{g}\-^{g}b_{h}\-^{gh}b_{k}b_{ghk}^{-1}$,
\end{center}
On the other hand,
\begin{center}
$\tilde\alpha(gh,k)\tilde\alpha(g,h)=(b_{gh}\-^{gh}b_{k}b_{ghk}^{-1})(b_{g}\-^{g}b_{h}b_{gh}^{-1})
=b_{g}\-^{g}b_{h}\-^{gh}b_{k}b_{ghk}^{-1}$.
\end{center}
Therefore, $\tilde\alpha$ is a factor set.
\par Next, we prove that $[\tilde\alpha]$ is independent of the choice of the continuous sections. Suppose that $s$ and $u$ are continuous sections for $\pi$. Set $b_{g}=s\alpha(g)$ and $b^{'}_{g}=u\alpha(g)$. Since $\pi(b_{g}^{'})=\alpha(g)=\pi(b_{g})$, then $b_{g}^{'}=b_{g}a_{g}$ for some $a_{g}\in A$. Obviously the function $\kappa:G\rightarrow A$, $g\mapsto a_{g}$, is continuous. Thus, \begin{center}
    $\bar\alpha(g,h)=b^{'}_{g}\-^{g}b^{'}_{h}b^{'}_{gh}=b_{g}\kappa(g)\-^{g}b_{h}\-^{g}\kappa(h)(\kappa(gh))^{-1}b^{-1}_{gh})$
\end{center}
\begin{center}
    =$(\kappa(g)^{g}\kappa(h)(\kappa(gh))^{-1})(b_{g}\-^{g}b_{h}b_{gh}^{-1})=\delta_{G}^{1}(\kappa)(g,h)\tilde\alpha(g,h),$
\end{center}
where $\delta_{G}^{1}(\kappa)(g,h)=\-^{g}\kappa(h)(\kappa(gh))^{-1}\kappa(g).$
Consequently, $\bar\alpha$ and $\tilde\alpha$ are cohomologous.
\\
\par Suppose that $(\alpha,r)$ and $(\beta,s)$ are cohomologous in $Der_{c}(G,(C,\lambda))$. Then, there is  $c\in C$ such that $\beta(g)=c^{-1}\alpha(g)^{g}c$, $\forall g\in G$.
\\Let $s:C\rightarrow A$ be a continuous section for $\pi$. Since

$$\pi(s(c^{-1}\alpha(g)^{g}c))=\pi(s(c)^{-1}s\alpha(g)^{g}s(c)),$$ then,  there exists a unique $\gamma(g)\in ker\pi=A$ such that $$\gamma(g)(s(c)^{-1}s\alpha(g)^{g}s(c))=s(c^{-1}\alpha(g)^{g}c).$$ It is clear that the map $\gamma:G\rightarrow A$, $g\mapsto \gamma(g)$ is  continuous. Therefore,\\
\\$\tilde\beta(g,h)=s\beta(g).^{g}s\beta(h).(s\beta(gh))^{-1}$\\  \\$=s(c^{-1}\alpha(g)^{g}c).^{g}s(c^{-1}\alpha(h)^{h}c).(s(c^{-1}\alpha(gh)^{gh}c))^{-1}$\\
\\$=\gamma(g)[s(c)^{-1}s\alpha(g)^{g}s(c)].^{g}(\gamma(h)[s(c)^{-1}s\alpha(h)^{h}s(c)])$\\
\\ $.(\gamma(gh)[s(c)^{-1}s\alpha(gh)^{gh}s(c)])^{-1}=\ ^{g}\gamma(h)\gamma(gh)^{-1}\gamma(g)
[s(c)^{-1}s\alpha(g)^{g}s(c)]$\\
\\$.^{g}[s(c)^{-1}s\alpha(h)^{h}s(c)].[s(c)^{-1}s\alpha(gh)^{gh}s(c)]^{-1}$\\
\\$=\delta_{G}^{1}(\gamma)(g,h)[s(c)^{-1}s\alpha(g)^{g}s\alpha(h)(a\alpha(gh))^{-1}s(c)]$\\
\\$=\delta_{G}^{1}(\gamma)(g,h)[s(c)^{-1}\delta^{1}(\alpha)(g,h)s(c)]=\delta_{G}^{1}(\gamma)(g,h)\tilde\alpha(g,h).$\\

The last equality is obtained from the fact that
$\tilde\alpha(g,h) \in Ker\mu$ and $s(c)\in B$. Now,
note that   $\tilde\alpha$ is cohomologous to
$\tilde\beta$, when $(\alpha,r)$ is cohomologous to $(\beta,s)$.
 Thus, $\delta^{1}$ is well-defined.
\\ \par Recall that a short sequence
\begin{center}
$1\to(A,\mathbf{1})\stackrel{\imath}\to (B,\mu)\stackrel{\pi}\to(C,\lambda)\to 1$
\end{center}
 of partially crossed topological $G-R$-bimodules is exact, if the following diagram is commutative and the top row is exact.
\begin{center}
$\xymatrix{1\ar[r]&A\ar[r]^{\imath}\ar[dr]_{\mathbf{1}}&B\ar[r]^{\pi}\ar[d]^{\mu}&C\ar[r]\ar[dl]^{\lambda}&1
\\ & & R & &}$
\end{center}
The above exact sequence is called a proper extension with continuous sections, if $1\to A\stackrel{\imath}\to B\stackrel{\pi}\to C\to 1$ is a proper extension and $\pi$ has a continuous section.
\begin{thm}\label{Theorem 3.1.} Let $1\to (A,\mathbf{1})\stackrel{\iota}\to (B,\mu)\stackrel{\pi}\to (C,\lambda)\to 1$ be a proper extension of partially crossed topological $G-R$-bimodules with continuous sections. Then, the following sequence is exact.
\begin{center}
$0\to H^{0}(G,A)\stackrel{\iota^{0}}\to H^{0}(G,B)\stackrel{\pi^{0}}\to H^{0}(G,C)\stackrel{\delta^{0}}\to H^{1}(G,A)\stackrel{\iota^{1}}\to $
\end{center}
 $ \ \ \ \ \ \ \ \ H^{1}(G,(B,\mu))\stackrel{\pi^{1}}\to H^{1}(G,(C,\lambda))\stackrel{\delta^{1}}\to H^{2}(G,A)$
 \end{thm}
\begin{proof} We may assume that $\iota$ is the inclusion map.
\par 1. By \cite[Theorem 4.1] {Sah1}, the sequence $$0\to H^{0}(G,A)\stackrel{\iota^{0}}\to H^{0}(G,B)\stackrel{\pi^{0}}\to H^{0}(G,C)\stackrel{\delta^{0}}\to H^{1}(G,A)$$ is exact.
\par 2. Exactness at $H^{1}(G,A)$: Let $c\in H^{0}(G,C)$. Then, there is $b\in B$ such that $\pi(b)=c$. So, $$\iota^{1}\delta^{0}(c)(g)=\iota(\delta^{0}(c)(g))=\iota(b^{-1}\-^{g}b)=b^{-1}\-^{g}b.$$ Consequently, $\iota^{1}\delta^{0}(c)\sim \mathbf{1}$. Conversely, let $[\alpha]\in Ker\iota^{1}$. Then, there is  $b\in B$ such that $\alpha(g)=b^{-1}\-^{g}b, \forall g\in G$. So, $\pi(b^{-1}$$^{g}b)=1, \forall g\in G$. Take $c=\pi(b)$. Hence, $c\in H^{0}(G,C)$. Thus, $\delta^{0}(c)\sim \alpha$.
\par 3. Exactness at $H^{1}(G,(B,\mu))$: Since $\pi^{1}\iota^{1}([\alpha])=\pi^{1}([(\alpha,1)]=[(\pi\circ\alpha,1)]=[(\mathbf{1},1)]=1$, then, $Im\iota^{1}\subset Ker\pi^{1}$. Conversely, let $[(\beta,r)]\in ker\pi^{1}$. Then, there is  $c\in C$ such that $\pi\beta(g)=c^{-1}$$^{g}c$, for all $g\in G$ and $r=\lambda(c)^{-1}z$, for some $z\in H^{0}(G,R)$. Let $b \in B$ and $c=\pi(b)$. Therefore, $\pi(\beta(g))=\pi(b^{-1}$$^{g}b)$, $\forall g\in G$. On the other hand, the map $\tau_{b}:A\rightarrow A$, $a\mapsto b^{-1}ab$, is a topological isomorphism, because $A$ is a normal subgroup of $B$. So, for every $g\in G$ there is a
 unique element $a_{g}\in G$ such that $\beta(g)=(b^{-1}a_{g}b)(b^{-1}$$^{g}b)$.
Thus, $\beta(g)=b^{-1}a_{g}\-^{g}b$, $\forall g\in G$.
 Hence, $a_{g}=b\beta(g)^{g}b^{-1}$, $\forall g\in G$.
Obviously, the map $\alpha:G\rightarrow A$, $g\mapsto a_{g}$, is a continuous crossed homomorphism,
and $\iota^{1}([\alpha])=[(\alpha,1)]=[(\alpha,z)]=[(\alpha,\lambda(c)r)]=[(\alpha,\mu(b)r]=[(\beta,r)]$.
\par 4. Exactness at $H^{1}(G,(C,\lambda))$: Let $[(\beta,r)] \in H^{1}(G,B)$ and $s$ be a continuous section for $\pi$. Then
$\delta^{1}\circ \pi^{1}([(\beta,r)])=[\widetilde{\pi\beta}]$.  There is a continuous map $z:G\to A$ such that $s\pi\beta(g)=\beta(g)z(g)$. Thus,
\begin{center}
$\widetilde{\pi\beta}(g,h)=s(\pi\beta(g))^{g}s(\pi\beta(h))(s(\pi\beta(gh)))^{-1}=\beta(g)^{g}\beta(h)\beta(gh)^{-1}\delta_{G}^{1}(z)(g,h)=\delta_{G}^{1}(z)(g,h).$
\end{center}
 So, $Im\pi^{1}\subset Ker\delta^{1}$. Conversely, let $[(\gamma,r)]\in ker\delta^{1}$. Then, there is a continuous function $\alpha:G\to A$ such that $\tilde{\gamma}=\delta_{G}^{1}(\alpha)$, where $[\tilde{\gamma}]=\delta^{1}([(\gamma,r)])$. Thus,
 \begin{center}
 $\tilde{\gamma}(g,h)=s\gamma(g)^{g}s\gamma(h)(s\gamma(gh))^{-1}=\-^{g}\alpha(h)\alpha(gh)^{-1}\alpha(g), \forall g,h\in G$.
 \end{center}
 Assume $\beta(g)=s\gamma(g)\alpha(g)^{-1}, \forall g\in G$. Since $A\subset Ker\mu$,
then $\beta$ is a continuous crossed homomorphism from $G$ to $B$, and $\pi\beta=\gamma$.  Also $(\beta,r)\in Der_{c}(G,(B,\mu))$ because $\mu\beta(g)=\mu(s\gamma(g)\alpha(g)^{-1})=\mu(s\gamma(g))=\lambda\gamma(g)=r^{g}r^{-1}$.
 Hence, $\pi^{1}([(\beta,r)])=[(\gamma,r)]$. This completes the proof.
 \end{proof}
\section{Principal homogeneous spaces over $(A,\mu)$ - a new
definition of $H^1(G, (A,\mu))$}\label{section 5}
Serre \cite{Ser} showed that if $A$ is a topological $G$-module in which $A$ is discrete and $G$ a profinite group then there is a bijection between the set, $P(A)$, of all classes of principal homogeneous spaces over  $A$ and $H^1(G, A)$. Similarly, Inassaridze \cite{Ina-h} (algebraically)  defined the $G$-torsor over a crossed $G$-module $(A,\mu)$ and showed that there is a natural isomorphism between the group, $E(G,A)$, of all  classes of $G$-torsors over $(A,\mu)$ and the Guin's first non-abelian cohomology group $H^{1}(G,(A,\mu))$ (see \cite[Theorem 4.2] {Ina-h}).
We will show that the first non-abelian cohomology of topological groups is closely related with principle homogeneous spaces.
\begin{defn}\label{ Definition 4.1.} A principal homogeneous space (or topological torsor) over  a partially crossed topological $G-R$-bimodule $(A,\mu)$  is a pair $(P,f)$ consisting of a $G$-space $P$,
on which $A$ acts on the right (in a manner compatible with $G$) so that for any $p\in P$ the natural map $e_{p}:A\to P$, $a\mapsto pa$ is a homeomorphism, and $f$ is a $G$-map from $P$ to $R$ such that $f(pa)=\mu^{-1}(a)f(p)$ for any $p\in P$, $a\in A$.
\end{defn}
\begin{defn}\label{ Definition 4.2.} It is said that  principal homogeneous spaces $(P,f)$ and $(Q,g)$ over a partially crossed topological $G-R$-bimodule $(A,\mu)$ are isomorphic if there is a homemorphism $\nu:P\to Q$ compatible with the actions of $G$ and $A$ such that $f(p)=g\circ\nu(p)$ mod $H^{0}(G,R)$ for any $p\in P$.
\end{defn}
 Obviously, the isomorphism in Definition \ref{ Definition 4.2.} is an equivalence relation in the set of all principle homogenous spaces over $(A,\mu)$. We denote by $\mathcal{P}(A,\mu)$ the set  of all
classes of principal homogeneous spaces over $(A,\mu)$. Suppose that $A$ acts on the right on itself by translations. Obviously, $(A,\mu^{-1})$ is a principal homogeneous space over $(A,\mu)$. We call it trivial topological torsor over $(A,\mu)$. Thus, $\mathcal{P}(A,\mu)\not=\emptyset$.
\begin{thm}\label{Theorem 4.3.} Let $(A,\mu)$ be a partially crossed topological $G-R$-bimodule. There is a bijection between $\mathcal{P}(A,\mu)$ and $H^{1}(G,(A,\mu))$.
\end{thm}
\begin{proof} If $[(P,f)]\in \mathcal{P}(A,\mu)$, we choose a point $p\in P$. If $g\in G$, one has $^{g}p\in P$, therefore
there exists a unique $\alpha_{p}^{P}(g)\in A$ such that $^{g}p=p\alpha_{p}^{P}(g)$. One can see that $g\mapsto \alpha_{p}^{P}(g)$ is a
continuous crossed homomorphism and $\mu\alpha_{p}^{P}(g)=f(p)f(p\alpha_{p}^{P}(g))^{-1}=f(p)f(^{g}p)^{-1}=f(p)^{g}f(p)^{-1}$. Thus, $(\alpha_{p}^{P},f(p))\in Der_{c}(G,(A,\mu))$. On the one hand, substituting $pa$ for $p$ changes this crossed homomorphism into $g\mapsto a^{-1}\alpha_{p}^{P}(g)^{g}a$, since $^{g}(pa)=\ ^{g}p^{g}a=p\alpha_{p}^{P}(g)^{g}a=(pa)a^{-1}\alpha_{p}^{P}(g)^{g}a$. Also, $f(pa)=\mu(a)^{-1}f(p)$. Therefore, $(\alpha_{q}^{P},f(q))\sim (\alpha_{p}^{P},f(p))$ for any $q\in P$. On the other hand, let $(P',f')$ be isomorphic to $(P,f)$. Thus, there is a homeomorphism $\nu:P\to P'$ with properties in Definition \ref{ Definition 4.2.}. Let $p'\in P'$ and $\nu(p)=p'$. Then $^{g}p'=\ ^{g}\nu(p)=\nu(^{g}p)=\nu(p\alpha_{p}^{P}(g))=\nu(p)\alpha_{p}^{P}(g)$. So, $\alpha_{p'}^{P'}=\alpha_{p}^{P}$. Therefore by Remark \ref{Remark 1.16.}, $ (\alpha_{p'}^{P'},f'(p'))\sim  (\alpha_{p}^{P},f(p))$. One may thus define $\lambda:\mathcal{P}(A,\mu)\to H^{1}(G,(A,\mu))$ via $[(P,f)]\mapsto [(\alpha_{p}^{P},f(p))]$.
\par Vise versa, one defines $\gamma:H^{1}(G,(A,\mu))\to \mathcal{P}(A,\mu)$ as follows:
\newline If $(\alpha,r)\in Der_{c}(G,(A,\mu))$, denote by $P_{\alpha}$ the group $A$ on which $G$ acts by the following \emph{twisted formula}:
$$ga=\alpha(g)^{g}a.$$
Let now $A$ act on the right on $P_{\alpha}$ by translations. Define $f_{r}:P_{\alpha}\to R$ by $f_{r}(a)=\mu^{-1}(a)r$. We have $$f_{r}(ga)=\mu^{-1}(\alpha(g)^{g}a)r=\ ^{g}\mu(a)^{-1}\mu^{-1}(\alpha(g))r=\ ^{g}(\mu(a)^{-1}r)=\ ^{g}f_{r}(a)$$ for any $g\in G$. Thus, $f_{r}$ is a $G$-map. In addition, for any $a,b\in A$, $$f_{r}(ab)=\mu(ab)^{-1}r=\mu(b)^{-1}(\mu(a)^{-1}r)=\mu(b)^{-1}f_{r}(a).$$ Therefore, $(P_{\alpha},f_{r})$ is a principal
homogeneous space over $(A,\mu)$.
\par If $(\alpha,r)\sim (\beta,s)$, then there is $a\in A$ such that $\beta(g)=a^{-1}\alpha(g)^{g}a$ for any $g\in G$, and $s=\mu(a)^{-1}rt$ for some $t\in H^{0}(G,R)$. Define $\nu:P_{\alpha}\to P_{\beta}$ by $p\mapsto a^{-1}p$. For every $g\in G, p\in P_{\alpha}$, then $$\nu(gp)=\nu(\alpha(g)^{g}p)=a^{-1}\alpha(g)^{g}p=\beta(g)^{g}{a^{-1}}^{g}p=\beta(g)^{g}(a^{-1}p)=g\nu(p).$$
Thus, $\nu$ is a $G$-map. Obviously, $\nu$ is compatible with the action $A$ on $P_{\alpha}$ and $P_{\beta}$.
\par For $p\in P_{\alpha}$, we get $$f_{s}(\nu(p))=f_{s}(a^{-1}p)=\mu(p^{-1}a)s=\mu^{-1}(p)\mu(a)\mu(a)^{-1}rt=\mu^{-1}(p)rt.$$
Therefore, $(P_{\alpha},f_{r})$ is isomorphic to $(P_{\beta},f_{s})$. Consequently, one can define $\gamma$ by $\gamma([(\alpha,r)])=[(P_{\alpha},f_{r})]$.
\par We will show that $\gamma\lambda=Id_{P(A,\mu)}$. Let $(Q,g)$ be a principle homogeneous space over $(A,\mu)$. Fix $q\in Q$. We define $\nu:Q\to P_{\alpha_{q}^{Q}}$ by $p\mapsto e_{q}^{-1}(p)$. Obviously $\nu$ is a homeomorphism. For any $h\in G$, $$\nu(^{h}p)=e_{q}^{-1}(^{h}p)=e_{q}^{-1}(^{h}(qe_{q}^{-1}(p))=e_{q}^{-1}(^{h}q^{h}e_{q}^{-1}(p))=\ ^{h}e_{q}^{-1}(p)=\ ^{h}\nu(p).$$ i.e., $\nu$ is a $G$-map. Also for any $a\in A$, $$\nu(pa)= e_{q}^{-1}(pa)=e_{q}^{-1}(qe_{q}^{-1}(p)a)=e_{q}^{-1}(p)a.$$ In addition for any $p\in P$, $$f_{g(q)}\circ\nu(p)=\mu^{-1}(\nu(p))g(q)=g(q\nu(p))=g(qe_{q}^{-1}(p))=g(p).$$ This implies that $\gamma\lambda=Id_{P(A,\mu)}$.
\par Conversely, let $(\beta,s)\in Der_{c}(G,(A,\mu))$. For every $p\in P_{\beta}$, $f_{s}(p)=\mu^{-1}(p)s$ and also for every $g\in G$, $p\alpha_{p}^{P_{\beta}}(g)=gp=\alpha(g)^{g}p$. Thus, $(\beta,s)\sim (\alpha_{p}^{P_{\beta}},f_{s}(p))$. This shows that $\lambda\gamma=Id_{H^{1}(G,(A,\mu))}$.
\end{proof}
\begin{rem}\label{last}
In Definition \ref{ Definition 4.1.}, we may consider the $G$-map $f:P\to R$ as a continuous one.
\end{rem}
\begin{noti}\label{ Notice 4.4.} If $Der_{c}(G,(A,\mu))$ induces naturally a group structure on $H^{1}(G,(A,\mu))$, then we will define a group product on $\mathcal{P}(A,\mu)$.
Let $[(P_{1},f_{1})], [(P_{2},f_{2})]\in \mathcal{P}(A,\mu)$. Fix $p_{1}\in P_{1}$ and $p_{2}\in P_{2}$. Set $P=A$ and let $A$ act on the right on $P$ by translations. Consider a new action of $G$ on $P$ by the formula:
$$ga=\ ^{f_{1}(p_{1})}\alpha_{p_{2}}^{P_{2}}(g)\alpha_{p_{1}}^{P_{1}}(g)^{g}a$$
Define the map $f:P\to R$ by $f(a)=\mu^{-1}(a)f_{1}(p_{1})f_{2}(p_{2})$. We will show that $(P,f)$ is a principle homogeneous space over $(A,\mu)$.
Using $\lambda$, $\gamma$ in the proof of Theorem \ref{Theorem 1.23.}, we consider the classes $[(\alpha_{p_{1}}^{P_{1}},f_{1}(p_{1})]$ and $[(\alpha_{p_{2}}^{P_{2}},f_{2}(p_{2})]$ corresponding to the classes $[(P_{1},f_{1})]$ and $[(P_{2},f_{2})]$, respectively. Since $H^{1}(G,(A,\mu))$ is group, then $$[(\alpha_{p_{1}}^{P_{1}},f_{1}(p_{1})][(\alpha_{p_{2}}^{P_{2}},f_{2}(p_{2})]=[(^{f_{1}(p_{1})}\alpha_{p_{2}}^{P_{2}}\alpha_{p_{1}}^{P_{1}},f_{1}(p_{1})f_{2}(p_{2}))].$$
Obviously $\gamma([(^{f_{1}(p_{1})}\alpha_{p_{2}}^{P_{2}}\alpha_{p_{1}}^{P_{1}},f_{1}(p_{1})f_{2}(p_{2}))])=[(P,f)]$. Thus, one can define a group product $\circ$ as follows: $$[(P_{1},f_{1})]\circ [(P_{2},f_{2})]=[(P,f)].$$
\end{noti}
\begin{cor}\label{Corollary 4.5.} Let $(A,\mu)$ be a partially crossed topological $G$-module. Then, $(\mathcal{P}(A,\mu),\circ)$ is a group.
\end{cor}
\begin{proof} It is clear that by Theorem \ref{Theorem 1.23.}, $H^{1}(G,(A,\mu))$ is a group and Notice \ref{ Notice 4.4.} implies that $\mathcal{P}(A,\mu)$ is a group by the action $\circ$.
\end{proof}


‎


\bigskip
\bigskip
{\footnotesize {\bf H.E. Koshkoshi}\; \\ {Department of  Mathematics, Faculty of Mathematical Sciences}, {University
of Guilan,} {Rasht, Iran.}\\
{\tt Email: h.e.koshkoshi@guilan.ac.ir}

{\footnotesize {\bf H. Sahleh}\; \\ {Department of  Mathematics, Faculty of Mathematical Sciences}, {University
of Guilan, P.O.Box 1914,} {Rasht, Iran.}\\
{\tt Email: sahleh@guilan.ac.ir}\\

\end{document}